\documentclass{amsart}
\usepackage{amssymb}

\newcommand{\ve}{\varepsilon}

\newcommand{\R}{\mathbb{R}}

\newcommand{\N}{\mathbb{N}}

\DeclareMathOperator{\vol}{vol}
\DeclareMathOperator{\ric}{Ric}
\DeclareMathOperator{\diam}{diam}

\usepackage[style=alphabetic,
            url=false,
            isbn=false,
            doi=false]{biblatex}
\begin{document}

\title{Integral Curvature Bounds and Betti Numbers}
\author{Runze Yu}
\address{Department of Mathematics, University of California, Los Angeles}
\email{yurunze2023@ucla.edu}
\begin{abstract}
    We introduce an upper bound of the Betti numbers of a compact Riemannian manifold given integral bounds on the average of the lowest eigenvalues of the curvature operator. We then establish a new curvature condition for the Betti numbers to vanish using the Bochner technique. This generalizes results from Gallot and Petersen-Wink.
\end{abstract}

\maketitle

\newtheorem{theorem}{Theorem}[section]
\newtheorem{lemma}[theorem]{Lemma}
\newtheorem{corollary}[theorem]{Corollary}
\newtheorem{remark}[theorem]{Remark}
\newtheorem{prop}[theorem]{Proposition}

\section{Introduction}
\subsection{Introduction}
The Bochner technique is concerned with the relationship between the curvature of a Riemannian manifold and its geometric and topological properties \cite{bochner}. The first applications of the Bochner technique are by Meyer \cite{meyer} and Berger \cite{berger} who proved vanishing results of Betti numbers on manifolds with positive curvature operator, and Meyer further showed that they are rational cohomology spheres. 

More recent works focused on the implication of integral curvature bounds. Gallot \cite{gallot} established finiteness and vanishing results when the norm of the negative part of the lowest eigenvalue of the Ricci curvature is bounded. Petersen-Wink \cite{Petersen} proved similar results using pointwise bound on the average of eigenvalues of the curvature operator. Additionally, Aubry established finiteness of $\pi_1$ and an upper bound of the diameter of a manifold in \cite{Aubry} using an integral bound of the lowest eigenvalue of the Ricci curvature pinched below a positive constant. Furthermore, when the diameter is also bounded from below, Aubry \cite{aubry2} proved that such manifolds are homeomorphic to $\mathbb S^n$. Using the same integral bound on the pinching of Ricci curvature, Petersen-Wei established results of volume-comparison in \cite{PeterWei}.

Specifically in this paper we study how integral lower bounds of the curvature impose finiteness and vanishing conditions on the Betti numbers of a Riemannian manifold. The result in positive curvature case is well-known, see for example \cite{PeterBook}. More generally, Gallot proposed in \cite{gallot} two prominent questions:
\begin{enumerate}
    \item The Finiteness Theorem. If the negative part of the Riemannian curvature is controlled, what bound does it give to the Betti numbers? 
    \item The Pinching Theorem. If the diameter of the manifold $M$ is bounded and the Riemannian curvature goes a little below zero, when will $b_k(M)\le b_k(\mathbb T^n)$ for $1\le k\le n,$ where $\mathbb T^n$ is the $n$-dimensional torus? 
\end{enumerate}
Both integral bounds and pointwise bounds have been studied for these questions. Given a bound $D$ on the diameter of an $n$-dimensional Riemannian manifold $(M,g)$, Gallot in \cite{gallot} showed that $b_1(M)\le b_1(\mathbb T^n)$ when the negative part of the Ricci curvature is bounded and sufficiently small in an integral sense. Furthermore, if the integral of the sectional curvature $\sigma$ is also bounded and sufficiently small, then $b_i(M)\le b_i(\mathbb T^n)$ for all $i.$ 

Petersen-Wink in \cite{Petersen} showed that $$b_1(M)=\cdots=b_k(M)=b_{n-k}(M)=\cdots =b_{n-1}(M)=0$$ in an $n$-dimensional Riemannian manifold $M$ with an $(n-k)$-positive Riemannian curvature for any $k\le \lfloor n/2\rfloor.$ Furthermore, if there is a pointwise bound on the average of lowest $(n-k)$ eigenvalues of the curvature operator and a bound $D$ on the diameter of the manifold, then $k$-th Betti number is bounded from above.

In this paper we generalize both \cite{gallot} and \cite{Petersen} by considering integral bounds on the average of lowest eigenvalues of the curvature tensor. Our first theorem demonstrates a bound on the $k$-th Betti number given an integral bound on the average of the lowest eigenvalues of the curvature tensor.
\begin{theorem}
\normalfont 
\label{main}
Let $(M,g)$ be a compact Riemannian manifold of dimension $n$ satisfying the Sobolev inequality:
$$
\|f\|_{2\nu} \le \|f\|_2+S\|df\|_2
$$
for some $\nu>1$ and the Sobolev constant $S.$ Let $1\le k\le n.$ Let the eigenvalues of the curvature operator $\mathcal R:\bigwedge^2TM\to\bigwedge^2TM$ be 
$$\lambda_1\le \lambda_2\le\cdots\le\lambda_{\binom n2}$$
and for some integer $1\le C\le \binom n2$ set 
$$\tilde\lambda=\dfrac1C\sum_{i=1}^C\lambda_i:M\to \R.$$
Let $\lambda=\|\tilde\lambda\|_{p'}$ and $p$ be the H\"older conjugate of $p',$ and $p\in(1,\nu).$ Then 
$$b_k(M)\le \binom nk\exp\bigl( \dfrac{Sp\sqrt{\lambda\nu}}{\sqrt{\nu}-\sqrt p}\bigr)$$
for any $p\in(1,\nu).$
\end{theorem}
When $p=1$ and $\lambda=\|\tilde\lambda\|_{p'}=\sup\tilde\lambda,$ Theorem \ref{main} can be found in \cite[Theorem 9.3.2]{PeterBook} with proof based on techniques developed in \cite{gal2}. In this paper, we develop a generalization of the Moser iteration which controls the supremum norm of a continuous function by its lower $L^p$-norms, given the Sobolev inequality on a manifold. The theorem follows from Peter Li's estimation of the dimension of a fiber bundle on the manifold by norms on the bundle, see for example \cite[Lemma 9.2.6]{PeterBook} and \cite{berard}. 

The Sobolev constant is bounded from above in \cite{gallot} if the diameter of the manifold is bounded and the Ricci curvature is bounded from below. 
\begin{prop}
\normalfont
\label{sobolev}
Let $(M,g)$ be a compact Riemannian manifold of dimension $n$ and let $\diam M$ be the diameter of $M.$ Suppose the manifold satisfies $\diam M<D$ and $\ric \ge (n-1)k+\ve(x)$ for some positive constant $D$ and $k.$ Then when the $L^p$-norm of the error term $\|\ve\|_p$ is sufficiently small, there is a constant $C=C(n,kD^2)$ such that $S\le D\cdot C(n,kD^2).$
\end{prop}
When the average of the lowest $C$ eigenvalues of the curvature operator $\mathcal R$ is considered as in the settings of Theorem \ref{main}, we shall see that the condition on the lower bound of the Ricci curvature is automatically satisfied. Thus an upper bound on the diameter gives the following corollary, which generalizes \cite[Theorem 10]{gallot}.
\begin{corollary}
\normalfont
\label{cor}
Let $(M,g)$ be a compact Riemannian manifold of dimension $n$ and let $D=\diam M.$ Then there is an $\ve=\ve(D,n,p)>0$ such that $b_k(M)\le b_k(\mathbb T^n)$ for any $1\le k\le n$ whenever $\lambda<\ve.$ 
\end{corollary}
The previous corollary is a pinching result using integral bounds on the curvature operator, where the curvature operator is pinched below 0. It is then natural to ask what conditions are necessary for the Betti numbers to vanish when the curvature operator is pinched below a positive constant, namely when the curvature operator is bounded below by 1 plus an error term $\ve.$ 
\begin{theorem}
\normalfont
Let $(M,g)$ be a compact Riemannian manifold of dimension $n$ and bounded diameter $\diam M<D.$ For a fixed $p\in(1,\nu),$ set $\|1-\tilde\lambda\|_{p'}= \kappa$
where $1/p+1/p'=1.$ Set
$$C(S,\kappa,p) := \exp\bigl( \dfrac{2pS\sqrt{\kappa\nu}}{\sqrt{\nu}-\sqrt p}\bigr)\kappa.$$
Then $b_k(M)=0$ if $\kappa < 1$ and $C(S,\kappa,p) < 1.$ Specifically, there is $\ve=\ve(D,n,p)<1$ such that $b_k(M)=0$ whenever $\kappa<\ve.$
 
\label{vanish}
\end{theorem}
We modify the Moser iteration above to derive a generalization of the maximum principle given an integral bound between a function $f$ and its derivative $\Delta f.$ Theorem \ref{vanish} now follows from an application of the Bochner technique for harmonic forms.

Futhermore, Aubry \cite{Aubry} established an upper bound of the diameter when the integral of the Ricci curvature pinched below a positive constant is bounded. A renormalization of \cite[Theorem 1.2]{Aubry} shows

\begin{prop}
\label{aubry}
\normalfont Let $(M, g)$ be a complete Riemannian manifold of dimension $n.$ For any $p>n/2$ and $k>0$ there exists a constant $C(p,n, k)$ such that if 
$$
\rho_p:=\int_M (\ric_{1} - k(n-1))_-^p < \dfrac{\vol M}{C(p,n,k)},
$$
where $\ric_1$ is the lowest eigenvalue of the Ricci curvature, then $M$ is compact and its diameter satisfies
$$
\diam M \le \dfrac\pi{\sqrt{k}}\bigl(1+C(p,n,k)\bigl(\dfrac{\rho_p}{\vol M}\bigr)^{1/10}\bigr)\le \dfrac\pi{\sqrt{k}}\bigl(1+C(p,n,k)^{9/10}\bigr).
$$
\end{prop}
Under the assumption of Proposition \ref{aubry}, Theorem \ref{vanish} can be further improved so that the bound only depends on the dimension.
\begin{corollary}
\label{cor2}
\normalfont
Let $(M,g)$ be a complete Riemannian manifold of dimension $n.$ For a fixed $p>\max\{n/2,\nu/(\nu-1)\}$ set $\kappa = \|1-\tilde\lambda\|_p.$ Then there is $\ve=\ve(n,p)$ such that $b_k(M)=0$ whenever $\kappa<\ve.$
\end{corollary}
\subsection{Acknowledgements.} The author would like to thank Professor Peter Petersen for introducing this question to the author and his support of this work. This project was completed as a part of UCLA summer REU 2022 under the supervision of Professor Peter Petersen.
\section{Bounding Betti Numbers}
\noindent
\subsection{Norms on the Space of Continuous Functions.} For a compact Riemannian manifold $(M,g)$ and any continuous $f:M\to\R$ we define
$$
\|f\|_{\infty}=\max_{x\in M}|f(x)|
$$
and for any $p\ge 1$
$$
\|f\|_p:=\bigl(\dfrac1{\vol M}\int_M|f|^pdg\bigr)^{1/p}.
$$
These are norms on the vector space of continuous functions on $M,$ and $\|f\|_p$ increases to $\|f\|_{\infty}$ as $p\to+\infty.$
\subsection{The Weitzenb\"ock Curvature} The \textit{Weitzenb\"ock curvature operator} on a tensor $T$ can be defined by
$$
\ric(T)(X_1,\cdots,X_k)=\sum (R(e_j,X_i)T)(X_1,\cdots, e_j, X_k).
$$
The symbol $\ric$ is used since the Weitzenb\"ock curvature is the Ricci curvature when evaluated on 1-forms and vector fields. The Weitzenb\"ock formula states that
$$
\triangle \omega = \nabla ^*\nabla \omega +\ric(\omega)
$$
where $\triangle =d\delta+\delta d$ is the Hodge Laplacian. Generally the right hand side of the Weitzenb\"ock formula is a \textit{Lichnerowicz Laplacian} defined by
$$
\Delta_LT:=\nabla ^*\nabla T +c\ric(T)
$$
for tensor $T$ and some $c>0.$ For a more complete exposition on the relationship between Lichnerowicz Laplacians and the Bochner technique, see \cite[Chapter 9]{PeterBook}.
\subsection{Controlling Betti Numbers} The following Lemma generalizes the classical Moser iteration.
\begin{lemma}
\label{moser}
\normalfont
Let $(M,g)$ be a compact Riemannian manifold such that 
$$\|u\|_{2\nu}\le S\|\nabla u\|_2+\|u\|_2$$
for some $\nu>1$ and a Sobolev constant $S.$ Suppose $f:M\to[0,+\infty)$ is continuous, smooth on $\{f>0\}$ and $$\Delta f\ge -\tilde\lambda f$$ where $\tilde\lambda:M\to\R$ is continuous. For a fixed  $p\in(1,\nu)$ let   
$
 \|\tilde\lambda\|_{p'}= \lambda
$
where $1/p+1/p'=1.$ Then 
$$
\|f\|_{\infty}\le \exp\bigl( \dfrac{S\sqrt{\lambda\nu}}{\sqrt{\nu}-\sqrt p}\bigr)\|f\|_{2p}.
$$
\end{lemma}
\begin{proof}
By Green's formula,
$$(f^{2q-1},\Delta f)=-(df^{2q-1},df)=-(2q-1)(f^{2q-2}df,df).$$
Then by H\"older's inequality,
\begin{align*}
    \|df^q\|_2^2&=q^2(f^{2q-2}df,df)=-\dfrac{q^2}{2q-1}(f^{2q-1},\Delta f) \\
    &\le \dfrac{q^2}{(2q-1)(\vol M)}\int_{M}f^{2q-1}\tilde\lambda fdg \\
    &=\dfrac{q^2}{(2q-1)(\vol M)}\int_{M}f^{2q}\tilde\lambda  dg \\
    &\le\dfrac{q^2}{(2q-1)(\vol M)}\bigl(\int_{M}f^{2pq}dg\bigr)^{1/p}\bigl(\int_{M}|\tilde\lambda |^{p'} dg\bigr)^{1/p'} \\
    &= \dfrac{q^2\lambda\|f\|_{2pq}^{2q}}{2q-1}.
\end{align*}
The Sobolev inequality implies 
\begin{align*}
    \|f^q\|_{2\nu}&\le S\|df^q\|_2+\|f^q\|_2 \\
    &\le \|f^q\|_2+\dfrac{Sq\sqrt\lambda\|f\|_{2pq}^{q}}{\sqrt{2q-1}} \\
    &\le \bigl(1+\dfrac{Sq\sqrt\lambda}{\sqrt{2q-1}}\bigr)\|f\|_{2p q}^q
\end{align*}
and then
$$
\|f\|_{2\nu q}\le \bigl(1+\dfrac{Sq\sqrt\lambda}{\sqrt{2q-1}}\bigr)^{1/q}\|f\|_{2p q}.
$$
Set $q=(\nu/p)^k$ for each $k\in\N.$ By iterating from $k=0,$ we can obtain
$$
\|f\|_{\infty}\le\|f\|_{2p}\prod_{k=0}^{+\infty}\bigl(1+\dfrac{S(\nu/p)^k\sqrt\lambda}{\sqrt{2(\nu/p)^k-1}}\bigr)^{1/(\nu/p)^k}.
$$
The infinite product can be estimated by taking logarithms and using $\log(1+x)\le x.$
\begin{align*}
    &\sum_{k=0}^{+\infty}\bigl(\dfrac\nu p\bigr)^{-k}\log \bigl(1+\dfrac{S(\nu/p)^k\sqrt\lambda}{\sqrt{2(\nu/p)^k-1}}\bigr) \\
    \le \, & \sum_{k=0}^{+\infty}\dfrac{S\sqrt\lambda}{\sqrt{2(\nu/p)^k-1}}\\
    \le\,& \sum_{k=0}^{+\infty}\dfrac{S\sqrt\lambda}{(\nu/p)^{k/2}} \\
    =\, & \dfrac{S\sqrt{\lambda\nu}}{\sqrt{\nu}-\sqrt p}.
\end{align*}
The conclusion follows.
\end{proof}

\begin{lemma}
\label{eigen}
\normalfont
Let $(M,g)$ be a compact Riemannian manifold and let $\mathcal R$ be the curvature operator viewed as an operator on $\bigwedge^2TM.$ Let $\lambda_1\le \lambda_2\le\cdots\le\lambda_{\binom n2}$ be eigenvalues of $\mathcal R.$ Suppose for a continuous $\tilde\lambda:M\to\R$ and some $1\le C\le \binom n2$ there is 
$$\dfrac1{C}(\lambda_1+\cdots+\lambda_C)\ge -\tilde\lambda.$$
Then $g(\ric(\omega),\omega)\ge -\tilde\lambda|\omega|^2$ where $\ric$ is the Weitzenb\"ock curvature.
\end{lemma}
\begin{proof}
This follows purely from pointwise algebraic manipulations, for example in \cite{Petersen}. Choose an orthonormal basis $\{\Xi_\alpha\}$ such that $\mathcal R(\Xi_\alpha)=\lambda_\alpha\Xi_\alpha.$ Then we can calculate that
\begin{align*}
    g(\ric(\omega),\omega) &=\sum_{\alpha=1}^{\binom n2}\lambda_\alpha|\Xi_\alpha \omega|^2 \\
    &=\sum_{\alpha=1}^{C}\lambda_\alpha|\Xi_\alpha \omega|^2+\sum_{\alpha=C+1}^{\binom n2}\lambda_\alpha|\Xi_\alpha \omega|^2 \\
    &\ge \sum_{\alpha=1}^{C}\lambda_\alpha|\Xi_\alpha \omega|^2+\lambda_{C+1}\sum_{\alpha=C+1}^{\binom n2}|\Xi_\alpha \omega|^2 \\
    &=\lambda_{C+1}|\omega|^2+\sum_{\alpha=1}^{C}(\lambda_\alpha-\lambda_{C+1})|\Xi_\alpha \omega|^2 \\
    &\ge \lambda_{C+1}|\omega|^2+\dfrac1C\sum_{\alpha=1}^{C}(\lambda_\alpha-\lambda_{C+1})|\omega|^2 \\
    &=\dfrac{|\omega|^2}C\sum_{\alpha=1}^{C}\lambda_\alpha\ge -\tilde\lambda|\omega|^2.
\end{align*}
This concludes the proof.
\end{proof}
There is an immediate corollary that derives a lower bound on the Ricci curvature and therefore a bound on the Sobolev constant.
\begin{corollary}
\normalfont
Let $(M,g)$ be a compact Riemannian manifold with diameter $\diam M<D.$ Let $\mathcal R$ be the curvature operator viewed as an operator on $\bigwedge^2TM.$ Let $\lambda_1\le \lambda_2\le\cdots\le\lambda_{\binom n2}$ be eigenvalues of $\mathcal R.$ Suppose for a continuous $\tilde\lambda:M\to\R$ and some $1\le C\le n-1$ such that
$$\dfrac1{C}(\lambda_1+\cdots+\lambda_C)\ge -\tilde\lambda.$$
Then there is a function $C(n,D,\tilde\lambda)$ such that $S\le C(n,D,\tilde\lambda).$
\end{corollary}
\begin{proof}
This follows by specializing the Weitzenb\"ock curvature to the Ricci curvature and applying Proposition \ref{sobolev}.
\end{proof}

The following lemma is an application of Peter Li's estimate of dimensions of fiber bundles. The original theorem can be found in \cite{PeterBook} or \cite{berard}.
\begin{lemma}
\label{PeterLi}
\normalfont 
Let $(M,g)$ be a compact Riemannian manifold and let $E$ be a vector bundle where each fiber has dimension $m,$ such that each fiber is endowed with a smoothly varying inner product. Suppose $V$ is a finite dimensional subspace of the vector space of sections of $E.$ For a fixed $p\ge 1$ set 
$$
C_p(V)=\sup_{s\in V-0}\dfrac{\|s\|_\infty}{\|s\|_{2p}}.
$$
Then 
$$
\dim V\le m\cdot C_p(V)^p.
$$
\end{lemma}
\begin{proof}

Note first that 
$$
C_p(V)^p=\sup_{s\in V-0}\dfrac{\|s\|_\infty^p}{\|s\|_{2p}^p}=\sup_{s\in V-0}\dfrac{\|s^p\|_\infty}{\|s^p\|_{2}}.
$$
Thus for any $s\in V-0,$
$$\|s\|_\infty=\|(|s|^{1/p})^p\|_{\infty}\le C_p(V)^p\|(|s|^{1/p})^p\|_2=C_p(V)^p\|s\|_2.$$
The assertion now follows from the theorem by P. Li.
\end{proof}
\begin{theorem}
\label{bochner}
\normalfont 
Let $(M,g)$ be a compact Riemannian manifold. Suppose for any harmonic $k$-form $\omega$ such that $g(\ric(\omega),\omega)\ge -\tilde\lambda|\omega|^2,$ where $\ric$ is the Weitzenb\"ock curvature operator. Following the notations in Lemma \ref{moser}, we have
$$b_k(M)\le \binom nk\exp\bigl( \dfrac{Sp\sqrt{\lambda\nu}}{\sqrt{\nu}-\sqrt p}\bigr).$$
\end{theorem}
\begin{proof}
It follows from Hodge theory that $$b_k(M)=\dim\mathcal H^k(M).$$ Pick any $\omega\in\mathcal H^k(M)$ and let $f=|\omega|.$  Then $f$ is nonnnegative and smooth except possibly at points where $\omega=0,$ which are minimal points of $f.$ There is
$$
2fdf=df^2=2g(\nabla \omega,\omega)\le 2|\nabla \omega|f.
$$
This is Kato's inequality $df\le |\nabla \omega|.$ By Bochner formula,
\begin{align*}
    f\Delta f&=\dfrac12\Delta f^2-|df|^2=|\nabla \omega|^2-g(\nabla^*\nabla \omega,\omega)-|df|^2\ge g(\ric(\omega),\omega)\\
    &\ge -\tilde\lambda f^2
\end{align*}
where the first inequality follows from the Weitzenb\"ock formula
$$
\triangle \omega = (d\delta+\delta d)\omega =\nabla^*\nabla\omega + \ric(\omega).
$$
Then $\Delta f\ge -\tilde\lambda f.$ Lemma \ref{moser} gives
$$
\dfrac{\|\omega\|_{\infty}}{\|\omega\|_{2p}}=\dfrac{\|f\|_{\infty}}{\|f\|_{2p}}\le \exp\bigl( \dfrac{S\sqrt{\lambda\nu}}{\sqrt{\nu}-\sqrt p}\bigr).
$$
Then it follows from Lemma \ref{PeterLi} that
$$
b_k(M)=\dim\mathcal H^k(M)\le \binom nk\cdot \exp\bigl( \dfrac{Sp\sqrt{\lambda\nu}}{\sqrt{\nu}-\sqrt p}\bigr).
$$
\end{proof}
Theorem \ref{main} now follows from combining Lemma \ref{eigen} and Theorem \ref{bochner}.
\section{Vanishing Theorem}
The next lemma is a vanishing version of the Moser iteration above.
\begin{lemma}
\label{vanMoser}
\normalfont
Let $(M,g)$ be a compact Riemannian manifold such that 
$$\|u\|_{2\nu}\le S\|\nabla u\|_2+\|u\|_2$$
for all smooth functions where $\nu>1.$ Suppose $f:M\to[0,+\infty)$ is continuous, smooth on $\{f>0\}$ and $$\Delta f\ge (1-\ve)f$$ where $\ve:M\to\R$ is continuous. For a fixed $p\in(1,\nu)$ set $\|\ve\|_{p'}= \kappa$
where $1/p+1/p'=1.$ Set
$$C(S,\kappa) := \exp\bigl( \dfrac{2pS\sqrt{\kappa\nu}}{\sqrt{\nu}-\sqrt p}\bigr)\kappa.$$
Then $f$ is constantly zero if $\kappa < 1$ and $C(S,\kappa) < 1.$
\end{lemma}
\begin{proof}
The same approach is used as in Theorem \ref{moser}. By Green's formula,
$$(f^{2q-1},\Delta f)=-(df^{2q-1},df)=-(2q-1)(f^{2q-2}df,df).$$
Then by H\"older's inequality,
\begin{align*}
    \|df^q\|_2^2&=q^2(f^{2q-2}df,df)=-\dfrac{q^2}{2q-1}(f^{2q-1},\Delta f) \\
    &\le \dfrac{q^2}{(2q-1)(\vol M)}\int_{M}f^{2q-1}(\ve-1)fdg \\
    &=\dfrac{q^2}{(2q-1)(\vol M)}\int_{M}f^{2q}\ve dg-\dfrac{q^2\|f^q\|_2^2}{2q-1} \\
    &\le \dfrac{q^2}{(2q-1)(\vol M)}\bigl(\int_{M}f^{2pq}dg\bigr)^{1/p}\bigl(\int_{M}|\ve|^{p'} dg\bigr)^{1/p'}-\dfrac{q^2\|f^q\|_2^2}{2q-1} \\
    &= \dfrac{q^2\|f^{q}\|_{2p}^2\kappa}{2q-1}-\dfrac{q^2\|f^q\|_2^2}{2q-1} .
\end{align*}
There is $\Delta f\ge (1-\ve)f\ge -\ve f.$ It follows from Lemma \ref{moser} that
\[
\|f\|_{2pq} \le \|f\|_{\infty} \le \exp\bigl( \dfrac{S\sqrt{\kappa\nu}}{\sqrt{\nu}-\sqrt p}\bigr)\|f\|_{2p}.
\]
We set $q=p.$ It follows that
\[
\|df^p\|_2^2\le \dfrac{p^2}{2p-1}\bigl(\|f^p\|_{2p}^2\kappa-\|f^p\|_{2}^2\bigr)\le \dfrac{p^2\|f\|_{2p}^{2p}}{2p-1}\bigl(C(S,\kappa)-1\bigr).
\]
When $C(S,\kappa)<1,$ there is $df\equiv 0,$ so $f$ is constant. Then $(1-\ve)f\le \Delta f\equiv 0.$ When $\kappa <1,$ there is a point $x\in M$ such that $1-\ve(x)>0,$ so $f(x)=0$ and then $f\equiv 0.$
\end{proof}
\begin{theorem}
\normalfont
Let $(M,g)$ be a compact Riemannian manifold. Suppose for any harmonic $k$-form $\omega$ we have
$$g(\ric(\omega),\omega)\ge (1-\ve)|\omega|^2$$
where $\ve$ satisfies the conditions in Lemma \ref{vanMoser}. Then $b_k(M)=0.$
\label{vanishing}
\end{theorem}
\begin{proof}
We have 
$$b_k(M)=\dim\mathcal H^k(M)$$
as before. Let $\omega$ be a harmonic $k$-form and $f=|\omega|.$ Then $f$ is nonnegative and smooth except possibly at points where $\omega=0.$ Kato's inequality and Bochner formula implies
\begin{align*}
    f\Delta f&=\dfrac12\Delta f^2-|df|^2=|\nabla \omega|^2-g(\nabla^*\nabla \omega,\omega)-|df|^2\ge g(\mathcal R(\omega),\omega)\\
    &\ge (1-\ve)f^2.
\end{align*}
Lemma \ref{vanMoser} now asserts $f$ is constant. Then
$$
0=\Delta\bigl(\dfrac12|\omega|^2\bigr)=|\nabla\omega|^2+g(\mathcal R(\omega),\omega)\ge|\nabla\omega|^2+(1-\ve)|\omega|^2.
$$
Thus
$$
0\ge\int_M|\nabla\omega|^2dg+|\omega|^2\int|1-\ve|dg.
$$
Note that the $L^1$-norm of $1-\ve$ is positive since $1>\kappa \ge\|\ve\|_1.$ It follows that $\omega$ must be parallel. This proves the claim.
\end{proof}
Theorem 2 now follows from Lemma \ref{eigen} and Theorem \ref{vanishing}. The claim follows since $\lim_{\kappa\to0}C(S,\kappa)=0.$
\printbibliography

\end{document}